\documentclass[a4,12pt]{amsart}
\usepackage{amsmath,amsfonts,amssymb,amsthm,latexsym}
\usepackage{amscd,graphicx,stmaryrd,xy}
\usepackage[english]{babel}
\usepackage{ae,aecompl}
\usepackage{xcolor}

\xyoption{all}
\begin{document}
\baselineskip14.625pt
\newcommand{\nc}[2]{\newcommand{#1}{#2}}
\newcommand{\rnc}[2]{\renewcommand{#1}{#2}}
\rnc{\theequation}{\thesection.\arabic{equation}}
\def\note#1{{}}

\newtheorem{definition}{Definition $\!\!$}[section]
\newtheorem{proposition}[definition]{Proposition $\!\!$}
\newtheorem{lemma}[definition]{Lemma $\!\!$}
\newtheorem{corollary}[definition]{Corollary $\!\!$}
\newtheorem{theorem}[definition]{Theorem $\!\!$}
\newtheorem{example}[definition]{\sc Example $\!\!$}
\newtheorem{remark}[definition]{\sc Remark $\!\!$}

\nc{\beq}{\begin{equation}}
\nc{\eeq}{\end{equation}}
\rnc{\[}{\beq}
\rnc{\]}{\eeq}
\nc{\ba}{\begin{array}}
\nc{\ea}{\end{array}}
\nc{\bea}{\begin{eqnarray}}
\nc{\beas}{\begin{eqnarray*}}
\nc{\eeas}{\end{eqnarray*}}
\nc{\eea}{\end{eqnarray}}
\nc{\be}{\begin{enumerate}}
\nc{\ee}{\end{enumerate}}
\nc{\bd}{\begin{diagram}}
\nc{\ed}{\end{diagram}}
\nc{\bi}{\begin{itemize}}
\nc{\ei}{\end{itemize}}
\nc{\bpr}{\begin{proposition}}
\nc{\bth}{\begin{theorem}}
\nc{\ble}{\begin{lemma}}
\nc{\bco}{\begin{corollary}}
\nc{\bre}{\begin{remark}\em}
\nc{\bex}{\begin{example}\em}
\nc{\bde}{\begin{definition}}
\nc{\ede}{\end{definition}}
\nc{\epr}{\end{proposition}}
\nc{\ethe}{\end{theorem}}
\nc{\ele}{\end{lemma}}
\nc{\eco}{\end{corollary}}
\nc{\ere}{\hfill\mbox{$\Diamond$}\end{remark} }
\nc{\eex}{\hfill\mbox{$\Diamond$}\end{example}}
\nc{\bpf}{{\it Proof.~~}}
\nc{\epf}{\hfill\mbox{$\square$}\vspace*{3mm}}
\nc{\hsp}{\hspace*}
\nc{\vsp}{\vspace*}
\newcommand{\wegdamit}[1]{}

\nc{\ot}{\otimes}
\nc{\te}{\!\ot\!}
\nc{\bmlp}{\mbox{\boldmath$\left(\right.$}}
\nc{\bmrp}{\mbox{\boldmath$\left.\right)$}}
\nc{\LAblp}{\mbox{\LARGE\boldmath$($}}
\nc{\LAbrp}{\mbox{\LARGE\boldmath$)$}}
\nc{\Lblp}{\mbox{\Large\boldmath$($}}
\nc{\Lbrp}{\mbox{\Large\boldmath$)$}}
\nc{\lblp}{\mbox{\large\boldmath$($}}
\nc{\lbrp}{\mbox{\large\boldmath$)$}}
\nc{\blp}{\mbox{\boldmath$($}}
\nc{\brp}{\mbox{\boldmath$)$}}
\nc{\LAlp}{\mbox{\LARGE $($}}
\nc{\LArp}{\mbox{\LARGE $)$}}
\nc{\Llp}{\mbox{\Large $($}}
\nc{\Lrp}{\mbox{\Large $)$}}
\nc{\llp}{\mbox{\large $($}}
\nc{\lrp}{\mbox{\large $)$}}
\nc{\lbc}{\mbox{\Large\boldmath$,$}}
\nc{\lc}{\mbox{\Large$,$}}
\nc{\Lall}{\mbox{\Large$\forall\;$}}
\nc{\bc}{\mbox{\boldmath$,$}}
\nc{\ra}{\rightarrow}
\nc{\ci}{\circ}
\nc{\cc}{\!\ci\!}
\nc{\lra}{\longrightarrow}
\nc{\imp}{\Rightarrow}
\rnc{\iff}{\Leftrightarrow}
\nc{\inc}{\mbox{$\,\subseteq\;$}}
\rnc{\subset}{\inc}
\def\sw#1{{\sb{(#1)}}}
\newcommand{\Boxneu}{\square}
\def\tr{{\rm tr}}
\def\Tr{{\rm Tr}}
\def\st{\stackrel}
\def\<{\langle}
\def\>{\rangle}
\def\d{\mbox{$\mathop{\mbox{\rm d}}$}}
\def\id{\mbox{$\mathop{\mbox{\rm id}}$}}
\def\ker{\mbox{$\mathop{\mbox{\rm Ker$\,$}}$}}
\def\coker{\mbox{$\mathop{\mbox{\rm Coker$\,$}}$}}
\def\hom{\mbox{$\mathop{\mbox{\rm Hom}}$}}
\def\im{\mbox{$\mathop{\mbox{\rm Im}}$}}
\def\map{\mbox{$\mathop{\mbox{\rm Map}}$}}
\def\o{\sp{[1]}}
\def\mo{\sp{[-1]}}
\def\z{\sp{[0]}}
\def\lhom#1#2#3{{{}\sb{#1}{\rm Hom}(#2,#3)}}
\def\rhom#1#2#3{{{\rm Hom}\sb{#1}(#2,#3)}}
\def\lend#1#2{{{}\sb{#1}{\rm End}(#2)}}
\def\rend#1#2{{{\rm End}\sb{#1}(#2)}}
\def\Rhom#1#2#3{{{\rm Hom}\sp{#1}(#2,#3)}}
\def\Lhom#1#2#3{{{}\sp{#1}{\rm Hom}(#2,#3)}}
\def\Rrhom#1#2#3#4{{{\rm Hom}\sp{#1}\sb{#2}(#3,#4)}}
\def\Llhom#1#2#3#4{{{}\sp{#1}\sb{#2}{\rm Hom}(#3,#4)}}
\def\LRhom#1#2#3#4#5#6{{{}\sp{#1}\sb{#2}{\rm Hom}\sp{#3}\sb{#4}(#5,#6)}}
\def\khom#1#2{{{\rm Hom}(#1,#2)}}

\nc{\spp}{\mbox{${\cal S}{\cal P}(P)$}}
\nc{\ob}{\mbox{$\Omega\sp{1}\! (\! B)$}}
\nc{\op}{\mbox{$\Omega\sp{1}\! (\! P)$}}
\nc{\oa}{\mbox{$\Omega\sp{1}\! (\! A)$}}
\nc{\dr}{\mbox{$\Delta_{R}$}}
\nc{\dsr}{\mbox{$\Delta_{\Omega^1P}$}}
\nc{\ad}{\mbox{$\mathop{\mbox{\rm Ad}}_R$}}
\nc{\as}{\mbox{$A(S^3\sb s)$}}
\nc{\bs}{\mbox{$A(S^2\sb s)$}}
\nc{\slc}{\mbox{$A(SL(2,\C))$}}
\nc{\suq}{\mbox{$\cO(SU_q(2))$}}
\nc{\tc}{\widetilde{can}}
\def\slq{\mbox{$\cO(SL_q(2))$}}
\def\asq{\mbox{$\cO(S_{q,s}^2)$}}
\def\esl{{\mbox{$E\sb{\frak s\frak l (2,{\Bbb C})}$}}}
\def\esu{{\mbox{$E\sb{\frak s\frak u(2)}$}}}
\def\ox{{\mbox{$\Omega\sp 1\sb{\frak M}X$}}}
\def\oxh{{\mbox{$\Omega\sp 1\sb{\frak M-hor}X$}}}
\def\oxs{{\mbox{$\Omega\sp 1\sb{\frak M-shor}X$}}}
\def\Fr{\mbox{Fr}}
\nc{\p}{{\rm pr}}
\newcommand{\can}{\operatorname{\it can}}

\nc{\vare}{\varepsilon}
\nc{\ha}{\mbox{$\alpha$}}
\nc{\hb}{\mbox{$\beta$}}
\nc{\hg}{\mbox{$\gamma$}}
\nc{\hd}{\mbox{$\delta$}}
\nc{\he}{\mbox{$\varepsilon$}}
\nc{\hz}{\mbox{$\zeta$}}
\nc{\hs}{\mbox{$\sigma$}}
\nc{\hk}{\mbox{$\kappa$}}
\nc{\hm}{\mbox{$\mu$}}
\nc{\hn}{\mbox{$\nu$}}
\nc{\hl}{\mbox{$\lambda$}}
\nc{\hG}{\mbox{$\Gamma$}}
\nc{\hD}{\mbox{$\Delta$}}
\nc{\ho}{\mbox{$\omega$}}
\nc{\hO}{\mbox{$\Omega$}}
\nc{\hp}{\mbox{$\pi$}}
\nc{\hP}{\mbox{$\Pi$}}
\nc{\co}{\mathrm{co}}

\nc{\qpb}{quantum principal bundle}
\def\gal{-Galois extension}
\def\hge{Hopf-Galois extension}
\def\ses{short exact sequence}
\def\csa{$C^*$-algebra}
\def\ncg{noncommutative geometry}
\def\wrt{with respect to}
\def\Ha{Hopf algebra}

\def\C{{\Bbb C}}
\def\N{{\Bbb N}}
\def\R{{\Bbb R}}
\def\Z{{\Bbb Z}}
\def\T{{\Bbb T}}
\def\Q{{\Bbb Q}}
\def\cO{{\mathcal O}}
\def\cT{{\cal T}}
\def\cA{{\cal A}}
\def\cD{{\cal D}}
\def\cB{{\cal B}}
\def\cK{{\cal K}}
\def\cH{{\cal H}}
\def\cM{{\cal M}}
\def\cJ{{\cal J}}
\def\fB{{\frak B}}
\def\pr{{\rm pr}}
\def\ta{\tilde a}
\def\tb{\tilde b}
\def\td{\tilde d}
\def\tA{\tilde A}
\def\tB{\tilde B}
\def\tK{\tilde K}
\def\tF{\tilde F}

\newenvironment{pf}{\emph{Proof.}}{\hfill$\Box$\\[1mm]}

\newcommand{\A}{{\mathcal{A}}}
\newcommand{\B}{{\mathcal{B}}}
\newcommand{\I}{{\mathcal{I}}}

\def\qed{{$\Box$}\medskip}

\def\bA{\mathbb{A}}
\def\bC{\mathbb{C}}
\def\bN{\mathbb{N}} 
\def\bQ{\mathbb{Q}}
\def\bR{\mathbb{R}}  
\def\bT{{\mathbf T}} 
\def\bZ{\mathbb{Z}}
\def\bl{{\mathbf \lambda}}

\def\one{\mathbf{1}}
\def\bigone{\mathbf{1}}
\def\can{\mathrm{can}}

\def\A{{A}}
\def\B{B}
\def\C{\mathcal{C}}
\def\F{\mathcal{F}}
\def\sG{\mathcal{G}}
\def\H{{H}}
 \def\sP{\mathcal{P}}
\def\sC{\mathcal{C}} \def\sR{\mathcal{R}}
\def\sO{\mathcal{O}} \def\sT{\mathcal{T}}
\def\L{{L}}
\def\M{{M}}

\def\N{\mathcal{N}}
\def\Eacute{\mathrm{\acute{E}}}
\def\eacute{\mathrm{\acute{e}}}
\def\P{\Psi}
\def\R{\mathcal{R}}
\def\V{\mathcal{V}}
\def\X{\mathcal{X}}
\def\Y{\mathcal{Y}}
\def\Z{\mathcal{Z}}

\def\LM{\L(\M)}

\def\Ad{\mathrm{Ad}} \def\id{\mathrm{id}}
\def\Inv{\mathrm{Inv}} \def\tr{\mathrm{tr}}
\def\Re{\mathrm{Re\,}} \def\Aut{\mathrm{Aut\,}}
\def\can{\mathrm{can}}
\def\der{\partial}
\def\im{\mathrm{im}}
\def\ker{\mathrm{ker}}
\def\coker{\mathrm{coker}}
\def\Ext{\mathrm{Ext}}
\def\Tor{\mathrm{Tor}}
\def\Cotor{\mathrm{Cotor}}
\def\Lin{\mathrm{Lin}}
\def\rank{\mathrm{rank}}
\def\degree{\mathrm{degree}}
\def\integers{\mathrm{integers}}
\def\integers{\mathrm{integers}}
\def\for{\mathrm{for}}
\def\isom{\cong}
\def\linear{\mathrm{linear}}
\def\otherwise{\mathrm{otherwise}}
\def\case{\mathrm{case}}
\def\even{\mathrm{even}}
\def\odd{\mathrm{odd}}
\def\op{\mathrm{op}}
\def\Ob{\mathrm{Ob}}
\def\zvec{{}^\bZ {\bf Vec}}


\newcommand{\nn}{\nonumber}         
\newcommand{\eps}{\varepsilon}      
\newcommand{\cop}{\Delta}           
\newcommand{\lcross}{{>\!\!\!\triangleleft}}
\newcommand{\rcross}{{\triangleright\!\!\!<}}
\newcommand{\sbullet}{\raisebox{1pt}{{\tiny \mbox{${\ \bullet}$}}}\,}

\def\be{\begin{equation}}
\def\ee{\end{equation}}
\def\ot{\otimes}
\def\u{\underline}

\def\IZ{{\Bbb Z}}
\def\IC{{\Bbb C}}
\def\a{\alpha}
\def\b{\beta}
\def\slq2{SL_q (2)}
\def\s{\sigma}
\def\smod{\sigma_{mod}}
\def\rto{\rightarrow}

\def\tr{\triangleright}
\def\tl{\triangleleft}
\def\btr{\blacktriangleright}
\def\btl{\blacktriangleleft}

\def\half{{\frac{1}{2}}}

\def\nchoosek{\left(\begin{array}{cr} n\\ k \end{array} \right)}
\def\qchooser{\left(\begin{array}{cr} q\\ r \end{array} \right)}
\def\kchoosel{\left(\begin{array}{cr} k\\ l \end{array} \right)}

\def\sqbnchooser{\left[\begin{array}{cr} n\\ r \end{array} \right]}

\def\ctn{\mathrm{ctn}}
\def\some{\mathrm{some}}
\def\coH{\mathrm{coH}}

\def\zerob{{\overline{(0)}}}
\def\oneb{{\overline{(1)}}}
\def\twob{{\overline{(2)}}}

\def\zero{{(0)}}
\def\one{{{(1)}}}
\def\two{{{(2)}}}
\def\three{{{(3)}}}
\def\four{{{(4)}}}
\def\minusone{{(-1)}}

\def\sqone{{{[1]}}}
\def\sqtwo{{{[2]}}}

\def\tone{{\langle1\rangle}}
\def\ttwo{{\langle2\rangle}}

\def\oneb{{\overline{(1)}}}
\def\twob{{\overline{(2)}}}
\def\one{{(1)}}
\def\two{{(2)}}
\def\three{{(3)}}
\def\four{{{(4)}}}
\def\five{{{(5)}}}
\def\minusone{{(-1)}}

\def\utimes{\,\underline{\otimes}\,}
\def\ev{\mathrm{ev}}

\def\HA{H}
\def\HB{H}
\def\NT{A} 
\nc{\pt}{\hspace{1pt}}\nc{\npt}{\hspace{-1pt}}

\title{
BRAIDED  JOIN COMODULE ALGEBRAS OF GALOIS OBJECTS 
}

\author{  
Ludwik D\k abrowski  }
\address{
SISSA (Scuola Internazionale Superiore di Studi Avanzati), 
Via Bonomea 265, 34136 Trieste, Italy}
\email{
dabrow@sissa.it}

\author{ 
Tom Hadfield }
\address{  
G-Research, 
Whittington House, 19-30 Alfred Place, London WC1E 7EA, United Kingdom}
\email{Thomas.Daniel.Hadfield@gmail.com}

\author{ 
Piotr M.~Hajac }
\address{  
Katedra Metod Matematycznych Fizyki, Uniwersytet Warszawski, 
ul.\ Ho\.za 74, Warszawa, 00-682\hspace{-1pt} Poland\hspace{-1pt}
and\hspace{-1pt}  
Instytut\hspace{-1pt} Matematyczny, Polska\hspace{-1pt} Akademia\hspace{-1pt} Nauk, 
ul.\,\'Sniadeckich\hspace{-1pt} 8, Warszawa, 00-656 Poland}

\email{pmh@impan.pl}

\author{  
Elmar Wagner}
\address{Instituto de F\'isica y Matem\'aticas, 
Universidad Michoacana, 
Ciudad Universitaria, Edificio C-3, Morelia, C.P. 58040 Mexico}
\email{
elmar@ifm.umich.mx}

\subjclass[2010]{Primary  *****, Secondary **** }


\date{}


\maketitle

{\large
\begin{abstract}\normalsize
We construct the join of noncommutative Galois objects (quantum torsors) 
over a Hopf algebra~$H$. To ensure
that the join algebra enjoys the natural (diagonal) coaction of~$H$, 
we braid the tensor product of the Galois objects. Then we show that 
this coaction is principal. Our examples are built from the
 noncommutative torus with the natural free action of the classical
torus, and arbitrary anti-Drinfeld doubles of finite-dimensional  Hopf algebras.
The former yields a 
noncommutative deformation of a non-trivial torus bundle, and the latter 
a finite quantum covering.
\end{abstract}
}

{\baselineskip12pt\tableofcontents}

\section{Introduction and preliminaries}
\setcounter{equation}{0}

In algebraic topology, the join of topological spaces is a fundamental concept. In particular it is used in the celebrated Milnor's construction
of a universal principal bundle~\cite{m-j56}. 
A~noncomutative-geometric generalization of the $n$-fold join $G*\cdots*G$ of a compact Hausdorff topological group $G$, which is the
first step in Milnor's construction, was proposed in~\cite{dhh} with $G$ replaced by Woronowicz's compact quantum group~\cite{w-sl98}. 
Herein our goal  is to provide another noncomutative-geometric version of the join $G*G$ now with $G$ replaced by a quantum torsor.

Just as compact quantum groups are captured by cosemisimple Hopf algebras, quantum torsors are given as Galois objects~\cite{c-s98}, i.e.\ 
comodule algebras with free and ergodic coactions. In particular, every Hopf algebra is a Galois object with its coproduct taken as a coaction.
One can think of Galois objects over Hopf algebras as principal $G$-bundles over a one-point space. This point of view is not very interesting
in the classical setting, but in the noncommutative-geometric framework it unlocks a plethora of new possiblities. 
Among prime examples of quantum torsors is the noncommutative 2-torus~\cite{r-ma90} with the natural action of the classical 2-torus.

To make this paper self-contained and to establish notation and terminology, we begin by recalling the basics of classical joins,
Hopf-Galois coactions~\cite{ss05}, strong connections~\cite{bh04} and the Durdevic braiding. 
In \cite{d-m96}, Durdevic proved that the algebra structure
on the left hand side of the Hopf-Galois canonical map, that is induced from the tensor algebra on its right hand side, 
is given by a braiding generalizing a standard Yetter-Drinfeld braiding of Hopf algebras. This generalization  hides inside the natural
Yetter-Drinfeld module structure, which was earlier observed by Doi and Takeuchi~\cite{dt89} forsaking the braided algebra multiplication.
It is this multiplication that we use to define a braided join algebra.

Hopf-Galois coactions that admit a strong connection
are quantum-group versions of compact principal bundles. 
Therefore we refer to them as principal coactions. Section~\ref{two} contains
the main result of this paper establishing the principality of the natural coaction on our braided join algebra:

\noindent{\bf Theorem 2.5}\ 
\emph{
Let $H$ be a Hopf algebra with bijective antipode. Assume that  $A$ is a bicomodule algebra
and a left and right Galois object over $H$. 
Then the diagonal coaction on the $H$-braided join algebra $A*_HA$ is \emph{principal}.
Furthermore, the coaction-invariant subalgebra  is isomorphic to the unreduced suspension of~$H$.
}

The remaining part of the paper is devoted to examples. In Section~\ref{three},
we unravel the structure of the braided join  of the aforementioned noncommutative 2-torus with itself. One can
view it as a field of noncommutative 4-tori over the unit interval with some collapsing at the endpoints.
Since this join is a noncommutative deformation of a nontrivial 2-torus principal bundle into a 2-torus quantum
principal bundle, it fits perfectly into the new framework for constructing interesting spectral triples \cite{cm08}
proposed recently in~\cite{ds13,dsz14,dz}.

Anti-Drinfeld doubles  were discovered as a tool for describing anti-Yetter-Drinfeld modules~\cite{hkrs04a}.
They  are already right Galois objects over  Drinfeld double Hopf al\-geb\-ras~\cite{d-vg87}. Hence we only needed to invent  left coactions
commuting with  right coactions and making anti-Drinfeld doubles also  left Galois objects. This is our second main result contained
  in the final Section~\ref{four}:

\noindent{\bf Theorem 4.1}\ 
\emph{
Let $H$ be a finite-dimensional Hopf algebra. Then the anti-Drinfeld double $A(H)$ is a bicomodule algebra
and a left and right Galois object over the Drinfeld double~$D(H)$.
}

\noindent
Drinfeld doubles are finite-dimensional Hopf algebras, so that one can think of them as finite quantum groups, and about their braided join
as a finite quantum covering. For a commutative Drinfeld double of dimension~$n$, our join construction would yield
 the set of all  line segments joining every point  in $\{(0,1),...,(0,n)\}$ to every point in $\{(1,1),...,(1,n)\}$. Note that taking the Drinfeld
double of the group Hopf algebra of any finite non-abelian group (e.g., the group $S_3$ of permutations of 3 elements) would already yield a 
noncommutative example.  However, to exemplify the generality of our theory, we choose a finite-dimesional Hopf algebra with 
antipode whose square is not identity.

Since modules over anti-Drinfeld doubles serve as coefficients of Hopf-cyclic ho\-mo\-lo\-gy and cohomology \cite{hkrs04b},
we hope that the aforesaid additional structure on anti-Drinfeld doubles will be  useful in Hopf-cyclic theory. Also, there seems to be
a clear way to generalize our braided join construction to $n$-fold braided joins of principal comodule algebras, 
and  to replace the algebra $C([0,1])$ of all
complex-valued continuous functions on the unit interval  by any algebra with an appropriate ideal structure.
However, this is beyond the scope of this paper (see~\cite{dhw,ddhw}).

\subsection{Classical principal bundles from the join construction}

Let $I=[0,1]$ be the closed unit interval and let $X$ be a topological space. The \emph{unreduced suspension} $\Sigma X$
of $X$ is the quotient of $I\times X$ by the equivalence relation $R_S$ generated by
\begin{align}
&(0,x)\sim (0,x'),\qquad (1,x)\sim (1,x').
\end{align}
Now take another topological space $Y$ and, on the space $I\times X\times Y$, consider the 
 equivalence relation  $R_J$ given by
\begin{equation}
(0,x,y)\sim (0,x',y),\qquad (1,x,y)\sim (1,x,y').
\end{equation}
The quotient space 
$X * Y:=(I\times X\times Y)/R_J$ 
is called the \emph{join} of $X$ and $Y$. It resembles the unreduced suspension of $X\times Y$, 
but with only $X$ collapsed at 0, and only $Y$ 
collapsed at~1. 
\vspace*{-0mm}\begin{figure}[h]
\[
\includegraphics[width=105mm]{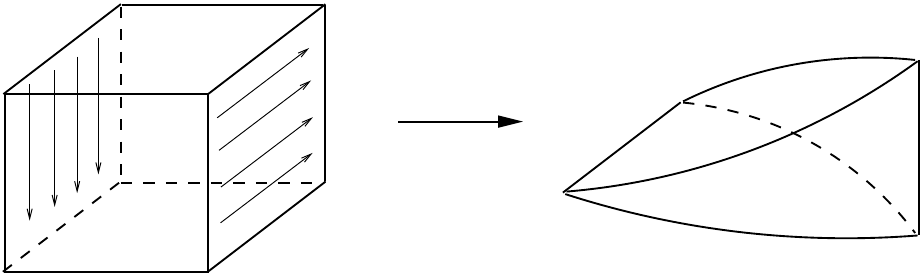}
\nonumber
\]
\end{figure}\vspace*{-0mm}

If $G$ is a topological group acting freely and continuously on  $X$ and $Y$,
then the diagonal \mbox{$G$-action} on $X\times Y$ induces a free continuous action on the join $X*Y$.
 Indeed,
the diagonal action of $G$ on $I\times X\times Y$ factorizes to 
the quotient, so that the formula
\[
([(t,x,y)],g)\longmapsto [(t,xg,yg)]
\]
makes $X*Y$ a right $G$-space. It is immediate that this action is free and continuous. 

On the other hand, let us take $X=Y$, and assume that we have a continuous map
$X\times X\stackrel{\phi}{\to} X$ such that for all $x\in X$ the maps
\[
X\ni y\longmapsto \phi(x,y)\in X\quad\text{and}\quad X\ni y\longmapsto \phi(y,x)\in X
\]
are homeomorphisms. Then, by \cite[Proposition~{VII.8.8}]{b-ge93}, the formula
\[\label{bredon}
\pi\colon X*X\ni [(t,x,y)]\longmapsto [(t,\phi(x,y))]\in\Sigma X
\]
defines a continuous surjection making the join $X*X$ a locally trivial fiber bundle over the unreduced suspension
$\Sigma X$ with the typical fiber~$X$.

In particular, we can combine the above described two cases of join constructions and take $X=G=Y$, 
where $G$ is a compact Hausdorff topological group. The diagonal action of $G$ on $G\times G$ yields
a free $G$-action on $G*G$ that is automatically proper due to the compactness of~$G$. Furthemore, taking
\[
\phi\colon G\times G\ni (g,h)\longmapsto gh^{-1}\in G,
\]
 we conclude that  $G*G$ is a locally trivial fiber bundle
over the unreduced suspension $\Sigma G$ with the typical fiber~$G$. Thus 
the join $G*G$ is a principal $G$-bundle with
\[
\pi\colon G*G\ni[(t,g,h)]\longmapsto [(t,gh^{-1})]\in\Sigma G.
\]
It is known that, since such a bundle is trivializable
if and only if $G$ is contractible, any non-trivial
compact Hausdorff topological group $G$ yields a non-trivializable
principal $G$-bundle over the unreduced suspension~$\Sigma G$.
For example, one can obtain in this way the
fibrations $S^7 \rightarrow S^4$, $S^3 \rightarrow S^2$   and 
$S^1 \rightarrow \mathbb{R}P^1$  
using $G=SU(2)$, $G=U(1)$ and  $\mathbb{Z}/2\mathbb{Z}$, respectively.

\subsection{Left and right Hopf-Galois coactions}

Let $\H$ be a Hopf algebra with coproduct~$\Delta$, counit~$\eps$ and 
antipode~$S$. Next, let $\Delta_P\colon P\to P\ot H$ be a coaction making $P$ a right \mbox{$H$-comodule} algebra,
and let ${}_Q\Delta\colon Q\to H\ot Q$ be a coaction making $Q$ a left $H$-comodule algebra. 
We shall frequently use
the Heyneman-Sweedler notation (with the summation sign suppressed) for coproduct and coactions:
\[
\Delta (h)=:h_{(1)}\ot h_{(2)}\,,\quad
\Delta _P(p)=:p_{(0)}\ot p_{(1)}\,,\quad
{}_Q\Delta (q) =:  q_\minusone \otimes q_\zero\,.
\]

Furthermore, let us define the coaction-invariant subalgebras:
\[
B:=P^{\mathrm{co}H} := \{ p \in P \; | \; \Delta _P (p) = p \otimes 1 \},\quad
D:={}^{\mathrm{co}H} Q := \{  q \in Q \; | \; {}_Q\Delta (q) = 1 \otimes q  \}.
\]
 We call a right (respectively left) coaction \emph{Hopf-Galois} \cite{ss05} iff the right (respectively left)  canonical map 
\begin{align}
\can_P : P \underset{B}{\otimes} P\ni p\otimes p'&\longmapsto pp'_{(0)}\otimes p'_{(1)}\in P \otimes H,
\\
{}_Q\can: Q \underset{D}{\otimes} Q\ni q \otimes q'&\longmapsto q_\minusone \otimes q_\zero q'\in H \otimes Q,
\label{can}\end{align} 
is a bijection. Observe that $\can_P$ is left linear over $P$ and right linear over $P^{\mathrm{co}H}$, whereas
${}_Q\can$ is left linear over ${}^{\mathrm{co}H} Q$ and right linear over~$Q$.

Now we  focus on left Hopf-Galois coactions. First, we define the left \emph{translation map}
\begin{equation}\label{ltrans}
\tau : H \longrightarrow Q \underset{D}{\otimes} Q, \quad \tau (h) := {}_Q\can^{-1} (h \otimes 1) 
=: h^\sqone \otimes  h^\sqtwo.
\end{equation} 
Note that, since ${}_Q\can$ is right $Q$-linear, so is ${}_Q\can^{-1}$. Therefore we obtain 
\[\label{leftq}
{}_Q\can^{-1} (h \otimes q) =h^\sqone \otimes  h^\sqtwo q.
\]
For the sake of clarity and completeness, herein we derive basic properties of the left translation map 
that are well known for the right translation map (the inverse of the right canonical map restricted to~$H$).
\bpr[cf.\ Remark 3.4 in \cite{s-hj90}]
Let ${}_Q\Delta\colon Q\to H\ot Q$ be a left Hopf-Galois coaction. Then,
for all $h,k\in H$ and $q\in Q$, the following equalities hold:
\begin{eqnarray}  
&&{q_\minusone}^\sqone \otimes {q_\minusone}^\sqtwo  q_\zero = q \otimes 1, \label{c}\\
&&  h^\sqone{}_{(-1)} \otimes h^\sqone{}_{(0)}\,  h^\sqtwo  =h \otimes 1, \label{SH}\\
&& h^\sqone\,h^\sqtwo = \vare(h), \label{ec}\\
&& (hk)^\sqone \otimes (hk)^\sqtwo 
= h^\sqone k^\sqone \otimes k^\sqtwo h^\sqtwo, \label{ac}\\
&& {h^\sqone}_\minusone \otimes {h^\sqone}_\zero \otimes h^\sqtwo    \label{lcl}
 = h_\one \otimes {h_\two}^\sqone \otimes {h_\two}^\sqtwo,\\
&&h^\sqone \otimes  {h^\sqtwo}_\minusone \otimes {h^\sqtwo}_\zero   \label{lacl}
= {h_\one}^\sqone \otimes S( h_\two ) \otimes {h_\one}^\sqtwo.
 \end{eqnarray}
 \epr
\begin{proof}
The  first identity \eqref{c} follows from \eqref{leftq} and ${}_Q\can^{-1}\circ {}_Q\can=\id$. The second equality
\eqref{SH} is an immediate
consequence of ${}_Q\can\circ {}_Q\can^{-1}=\id$.
Applying $\varepsilon\ot\id$ to \eqref{SH}  yields~\eqref{ec}. Since ${}_Q\can$ is injective, applying  it to both sides of \eqref{ac}, 
and using \eqref{SH} twice 
on the right hand side, proves \eqref{ac}. 
Transforming the left $H$-covariance of the canonical map ${}_Q\can$ 
\[
 (\id\ot{}_Q\can)   \circ ({}_Q\Delta \ot \id)   =  (\Delta \ot \id)\circ {}_Q\can
\]
to
\[
({}_Q\Delta \ot \id)\circ {}_Q\can^{-1}  =    (\id\ot{}_Q\can^{-1} )   \circ(\Delta \ot \id)
\]
we obtain the left $H$-covariance~\eqref{lcl}. 

To show the right $H$-covariance \eqref{lacl}, we apply the bijective map $(\id\ot \can_L)\circ (\text{flip}\ot \id)$ to both sides of~\eqref{lacl}. 
On the right hand side, we get 
\[
 S( h_\two )  \otimes {h_\one}^\sqone{}_{(-1)}  \otimes  {h_\one}^\sqone{}_{(0)}\,   {h_\one}^\sqtwo = 
 S( h_\two )  \otimes h_\one \ot 1.
\label{Sh1}
\]
Taking into account the left covariance \eqref{lcl}, the left hand side yields 
\[
{h^\sqtwo}_\minusone \otimes h^\sqone{}_{(-1)} \otimes h^\sqone{}_{(0)}\,{h^\sqtwo}_\zero =
{h_{(2)}{}^\sqtwo}{}_\minusone \otimes h_{(1)} \otimes h_{(2)}{}^\sqone\, {h_{(2)}{}^\sqtwo}{}_\zero . 
\] 
Thus \eqref{lacl} is equivalent to the equality
\[ 
S(h) \ot 1 =
{h^\sqtwo}_\minusone \otimes  h^\sqone\, {h^\sqtwo}_\zero.
\]
Finally, using \eqref{SH}, we compute
\begin{align}
{h^\sqtwo}_\minusone \otimes  h^\sqone\, {h^\sqtwo}_\zero 
&=\vare(h^\sqone{}_{(-1)}) \,{h^\sqtwo}_\minusone \otimes  h^\sqone{}_\zero\, {h^\sqtwo}_\zero \nonumber\\
&=S(h^\sqone{}_{(-2)}) \, h^\sqone{}_{(-1)} \,{h^\sqtwo}_\minusone \otimes  h^\sqone{}_\zero\, {h^\sqtwo}_\zero \nonumber\\ 
&= \big(S(h^\sqone{}_{(-1)})\ot 1\big)\,{}_Q\Delta (h^\sqone{}_\zero\, {h^\sqtwo})\nonumber\\
&= S(h) \ot 1
\end{align}
proving~\eqref{lacl}.
\end{proof}

\subsection{Principal right coactions}

Principal coactions are Hopf-Galois coactions with additional properties~\cite{bh04}.
 One can easily prove (see \cite[p.~599]{hkmz11} and references therein) that
a comodule algebra is principal if and only if it admits a strong 
connection. Therefore we will treat the existence of a strong connection 
as a condition defining the principality of a comodule algebra and avoid the original
definition of a principal coaction~\cite{bh04}. The latter is important when going beyond coactions that are
algebra homomorphisms --- then the existence of a strong connection is implied by
principality \cite{bh04} but we do not have the reverse implication.

\begin{definition}[\cite{bh04}]
Let $ H$ be a Hopf algebra with  bijective antipode.
A \emph{strong connection $\ell$} on $ P$ is a unital linear map $\ell : H \rightarrow  P \otimes  P$ satisfying:
\begin{enumerate}
\item 
$(\mathrm{id}\otimes \Delta_ P) \circ 
\ell = (\ell \otimes \mathrm{id}) \circ \Delta$,
$(\Delta_P^L \otimes \mathrm{id}) \circ 
\ell = (\mathrm{id} \otimes \ell) \circ
\Delta$, where 
$\Delta_P^L :=(S^{-1}\otimes\id)\circ\mathrm{flip}\circ\Delta_P$;
\item 
$\widetilde{\can} \circ \ell=1 \otimes \mathrm{id}$, where
$\widetilde{\can}\colon  P\otimes  P\ni p\otimes q\mapsto (p\ot 1)\Delta_P(q)\in  P\otimes H$.
\end{enumerate}
\end{definition}
\noindent
We will use the  Heyneman-Sweedler-type notation 
\begin{equation}\label{heysweed}
\ell(h)=:\ell(h)^{\langle 1 \rangle}\otimes 
 \ell(h)^{\langle 2 \rangle}=:h^{\langle 1 \rangle}\otimes 
 h^{\langle 2 \rangle}
\end{equation}
 with the summation sign suppressed. For the sake of brevity, we also suppress $\ell$ when it is clear which strong connection it refers to.

\subsection{Left Durdevic braiding}

Let ${}_Q\Delta\colon Q\to H\ot Q$ be a left Hopf-Galois coaction, and $D$ the coaction-invariant subalgebra. 
Using the bijectivity of the canonical map ${}_Q\can$,
we pullback the tensor algebra structure on $H\ot Q$ to $Q\ot_D Q$. The thus obtained algebra we shall denote by $Q\utimes_{{}_D} Q$
and call a \emph{left Hopf-Galois braided algebra}.
From the commutativity of the diagram
\[\label{multpull}
\xymatrix{ (Q\utimes_{{}_D} Q)\ot (Q\utimes_{{}_D} Q)\; \ar[r]^-{m_{Q\utimes_{{}_D} Q}} \ar[d]_{{}_Q\can\otimes{}_Q\can} & 
 \;Q\utimes_{{}_D} Q  \ar[d]^{{}_Q\can}\\
 (H \otimes  Q)\ot (H \otimes  Q)\ar[r]^-{m_{H\otimes Q}} & H \otimes  Q \,,
}
\]
we obtain the following explicit formula for the multiplication map $m_{Q\utimes_{{}_D} Q}$:
\begin{align}\label{multbrfor}
m_{Q\utimes_{{}_D} Q}(a\ot b\ot a'\ot b') &= 
{}_Q\can^{-1}\big({}_Q\can(a\ot b)\, {}_Q\can (a'\ot b')\big)\nonumber\\
&= {}_Q\can^{-1}\big(a_{(-1)}a'\!{}_{(-1)} \ot a_{(0)}b a'\!{}_{(0)} b')\big)\nonumber\\
&= a_{(-1)}{}^{[1]} a'\!{}_{(-1)}{}^{[1]} \ot a'\!{}_{(-1)}{}^{[2]} a_{(0)}{}^{[2]}  a_{(0)}b a'\!{}_{(0)} b'\nonumber\\
&= a\, a'\!{}_{(-1)}{}^{[1]} \ot a'\!{}_{(-1)}{}^{[2]} b a'\!{}_{(0)}\, b' .
\end{align}
Here in the last equality we used~\eqref{c}. 

Next, we show that  $m_{Q\utimes_{{}_D} Q}$  is the multiplication in a braided tensor  algebra associated to the left-sided version of
Durdevic's braiding \cite[(2.2)]{d-m96}. Since ${}_Q\can$ is left and right $D$-linear, the following formula defines a left and right
$D$-linear map:
\[\label{ldb}
Q\utimes_{{}_D} Q\ni x \otimes y\stackrel{\P}{\longmapsto} {}_Q\can^{-1}\big((1\ot x)\, {}_Q\can (y\ot 1)\big)=
{y_\minusone}^\sqone \otimes {y_\minusone}^{\sqtwo} x y_\zero\in 
Q\utimes_{{}_D} Q.
\]
Now we can write the multiplication formula \eqref{multbrfor} as
\begin{equation}
\label{braided_product}
m_{Q\utimes_{{}_D} Q}(a\ot b\ot a'\ot b')= a \P( b \otimes a' ) b'=:(a \otimes b) \sbullet ( a' \otimes b') .
\end{equation}
Note that when we view a Hopf algebra $H$ as a left comodule algebra  over itself, then the  left Durdevic braiding~\eqref{ldb}
becomes the Yetter-Drinfeld braiding: 
\begin{equation}
\label{ourbraid}
H\ot H\ni x \otimes y\longmapsto
 y_\one \otimes S( y_\two) x  y_\three\in H\ot H .
\end{equation}
\bpr [cf.\ Proposition~2.1 in \cite{d-m96}]
Let ${}_Q\Delta\colon Q\to H\ot Q$ be a left Hopf-Galois coaction, and $D$ the coaction-invariant subalgebra. Then
the map $\P$ defined in \eqref{ldb} is bijective and enjoys the following properties:
\begin{eqnarray}
m_Q \circ \P &=& m_Q\,,
\label{brai}\\
\forall \; q \in Q:\; \P( q \otimes 1) &=& 1 \otimes q, 
\label{braida}\\
\forall \; q \in Q:\; \P( 1 \otimes q) &=& q \otimes 1, 
\label{braidaa}\\
\P\circ  ( m_Q \otimes \id) &=& (\id \otimes m_Q)\circ  (\P \otimes \id) \circ (\id \otimes \P),
\label{braidb}\\
 \P\circ  (\id \otimes m_Q ) &=& ( m_Q \otimes \id)\circ  (\id \otimes \P)\circ (\P \otimes \id), 
\label{braidc}\\
(\P \otimes \id) \circ   (\id \otimes \P)\circ (\P \otimes \id)  &=&  (\id \otimes \P)\circ  (\P \otimes \id) \circ  (\id \otimes \P).
\label{braid}
\end{eqnarray} 
\epr 
%
%
\begin{pf} 
The bijectivity of $\P$ follows immediately from the fact that ${}_Q\can$ is an algebra isomorphism~\eqref{multpull}.
The braided commutativity of \eqref{brai} is a consequence of~\eqref{ec}. The condition \eqref{braida} is obvious,
 and the sibling condition \eqref{braidaa} is implied by~\eqref{c}. 

To prove \eqref{braidb}, using \eqref{lcl} and \eqref{ec}, we compute 
\begin{align}
&\big((\id \otimes m_Q)\circ(\P \otimes \id) \circ  (\id \otimes \P)\big)(x\ot y\ot z)\nonumber\\
&= \big((\id \otimes m_Q)\circ (\P \otimes \id)\big) (x\ot z_{(-1)}{}^\sqone \ot {z_\minusone}^{\sqtwo} \,y \, z_\zero) 
\nonumber\\
&= 
(\id \otimes m_Q)(z_{(-2)}{}^\sqone \ot z_{(-2)}{}^\sqtwo \, x \,z_{(-1)}{}^\sqone\ot z_{(-1)}{}^\sqtwo\, y \,z_\zero)\nonumber\\
&= z_{(-1)}{}^\sqone \ot z_{(-1)}{}^\sqtwo \, x \, y \,z_\zero \nonumber\\
&=\big(\P\circ(m_Q \ot\id)\big)(x\ot y\ot z) \,.
\label{brb}
\end{align}

Much in the same way,  to prove \eqref{braidc}, using \eqref{ac}, we compute 
\begin{align}
&\big(( m_Q \otimes \id)\circ(\id \otimes \P) \circ (\P \otimes \id) \big)(x\ot y\ot z)\nonumber\\
&=  \big(( m_Q \otimes \id)\circ(\id \otimes \P)\big) (y_{(-1)}{}^\sqone \ot {y_\minusone}^{\sqtwo} \,x \, y_\zero \ot z) 
\nonumber\\
&= \label{brc}
( m_Q \otimes \id)( y_{(-1)}{}^\sqone \ot z_{(-1)}{}^\sqone \ot z_{(-1)}{}^\sqtwo\, {y_\minusone}^{\sqtwo} \,x \, y_\zero\,  z_\zero)
 \nonumber\\
&={(yz)_{(-1)}}^\sqone \ot {(yz)_{(-1)} }^{\sqtwo} \,x \, (yz)_\zero\nonumber\\ 
& = \big(\P\circ  (\id \ot m_Q)\big)(x\ot y\ot z). 
\end{align}

Finally, to show \eqref{braid} first we apply ${}_Q\can\ot \id$ to its left hand side and,
taking advantage of the fact that above we have 
already computed $(\id \otimes \P) \circ (\P \otimes \id)$,
we proceed as follows:
\begin{align}
&\big(({}_Q\can\ot \id) \circ(\P \otimes \id) \circ   (\id \otimes \P)\circ (\P \otimes \id) \big)(x\ot y\ot z)\nonumber\\
&=
\big(({}_Q\can\ot \id) \circ(\P \otimes \id)\big)
( y_{(-1)}{}^\sqone \ot z_{(-1)}{}^\sqone \ot z_{(-1)}{}^\sqtwo\, {y_\minusone}^{\sqtwo} \,x \, y_\zero\,  z_\zero)\nonumber\\
&=
z_{(-2)}\ot y_{(-1)}{}^\sqone z_{(-1)}{}^\sqone \ot z_{(-1)}{}^\sqtwo\, {y_\minusone}^{\sqtwo} \,x \, y_\zero\,  z_\zero\,.
\end{align}
Here in the last equality we used~\eqref{ldb}.

Again much in the same way, taking advantage of the fact that above we have 
already computed $(\P \otimes \id) \circ  (\id \otimes \P)$, we apply ${}_Q\can\ot \id$ to the right hand side of \eqref{braid}, and
 proceed as follows:
\begin{align}
&
\big(({}_Q\can\ot \id)\circ  (\id \otimes \P)\circ  (\P \otimes \id) \circ  (\id \otimes \P) \big)(x\ot y\ot z)
\nonumber\\ &=
\big(({}_Q\can\ot \id) \circ(\id \otimes \P)\big)
(z_{(-2)}{}^\sqone \ot z_{(-2)}{}^\sqtwo \, x \,z_{(-1)}{}^\sqone\ot z_{(-1)}{}^\sqtwo\, y \,z_\zero)
\nonumber\\ &=
({}_Q\can\ot \id) \big(z_{(-4)}{}^\sqone \ot \tau\big(S(z_{(-2)}) y_{(-1)} \,z_{(-1)}\big)
z_{(-4)}{}^\sqtwo \, x \,z_{(-3)}{}^\sqone z_{(-3)}{}^\sqtwo\, y_\zero \,z_\zero \big)
\nonumber\\ &= 
({}_Q\can\ot \id) \big(z_{(-3)}{}^\sqone \ot 
\tau\big(S(z_{(-2)}) y_{(-1)} \,z_{(-1)} \big)
z_{(-3)}{}^\sqtwo \, x \, y_\zero \,z_\zero \big)
\nonumber\\ &= 
z_{(-4)} \ot   z_{(-3)}{}^\sqone \big(S(z_{(-2)}) y_{(-1)} \,z_{(-1)} \big){}^\sqone
\ot
\big(S(z_{(-2)}) y_{(-1)} \,z_{(-1)} \big){}^\sqtwo  \,
z_{(-3)}{}^\sqtwo \, x \, y_\zero \,z_\zero 
\nonumber\\ &= 
z_{(-4)} \ot   \big(z_{(-3)}S(z_{(-2)}) y_{(-1)} \,z_{(-1)} \big){}^\sqone
\ot
\big(z_{(-3)}S(z_{(-2)}) y_{(-1)} \,z_{(-1)} \big){}^\sqtwo  \,
x \, y_\zero \,z_\zero
\nonumber\\ &= 
z_{(-2)} \ot   y_{(-1)}{}^\sqone \,z_{(-1)}{}^\sqone \ot
z_{(-1)}{}^\sqtwo y_{(-1)}{}^\sqtwo \, x \, y_\zero \,z_\zero \,.
\end{align}
Here  we consecutively used \eqref{lacl}, \eqref{ec}, \eqref{lcl} and  \eqref{ac}. 
 Since ${}_Q\can\ot \id$ is bijective, this proves~\eqref{braid}.
\end{pf}

\section{Braided  principal join comodule algebras}\label{two}
\setcounter{equation}{0}

\subsection{Left braided right comodule algebras}

Now we shall consider left and right coactions simultaneously. Let $A$ be an $H$-bicomodule algebra, i.e.\ a left
and right $H$-comodule algebra with commuting coactions:
$({}_A\Delta\ot \id)\circ \Delta_A = (\id\ot \Delta_A) \circ {}_A\Delta$.
This coassociativity allows us to use the Heyneman-Sweedler notation over integers:
\[\label{bicom}
\big(({}_A\Delta\ot \id)\circ \Delta_A\big)(a) =a_{(-1)}\ot a_{(0)}\ot a_{(1)} =\big((\id\ot \Delta_A) \circ {}_A\Delta\big)(a).
\]

\begin{lemma} \label{unbraid}
Let $H$ be a Hopf algebra and  $A$ be a bicomodule algebra over $H$. Also, assume that the left coaction is Hopf-Galois, and that
the left and right coaction-invariant subalgebras coincide: 
${}^{\mathrm{co}H}\!A=A^{\mathrm{co}H}=:B$.
Let 
 $A \, \utimes_{{}_B} \, A$ be a left Hopf-Galois braided algebra.
Then the left canonical map \eqref{can}
is an isomorphism of right $H$-comodule algebras intertwining the  coactions given by the formulas
\vspace*{-2mm}\begin{align*}
&\Delta_{A\utimes_{{}_B}A}(a\utimes b):=a_{(0)}\utimes b_{(0)}\ot a_{(1)}b_{(1)}\,,\\
& \Delta_{H\ot A} (h\ot a):=(\id\ot\Delta_A)(h\ot a)= h\ot a_\zero \ot a_\one\,.
\end{align*}
\end{lemma}  
\begin{proof}
To verify the commutativity of the diagram
\[\label{comultpull}
\xymatrix{ 
A\utimes_{{}_B} A
\; \ar[r]^-{\Delta_{A\utimes_{{}_B} A}} \ar[d]_{{}_A\can} 
& 
 \;(A\utimes_{{}_B} A)\ot H  \ar[d]^{{}_A\can\ot\id}\\
 H \otimes  A\ar[r]^-{\Delta_{H\otimes A}} & (H \otimes A) \ot H \,,
}
\]
 for any $a,a'\in A$, using \eqref{bicom}, we compute:
\begin{align}
\big(( {}_A\can \ot\id)\circ\Delta_{A\utimes_{{}_B}A}\big)(a\utimes a')&= ({}_A\can\ot\id)(a_{(0)}\utimes a'_{(0)} \ot a_{(1)}  a'_{(1)})
\nonumber\\
&=a_{(-1)}\ot a_{(0)} a'\!{}_{(0)} \ot a_{(1)}  a'\!{}_{(1)}
\nonumber\\
&=(\id\ot\Delta_A)(a_{(-1)}\ot a_{(0)} a')
\nonumber\\
&=(\Delta_{H\ot A}\circ {}_A\can)( a \utimes a').
\end{align}
This shows that ${}_A\can$ is right $H$-colinear. Also, since ${}_A\can$ and $\Delta_{H\ot A}$
are algebra homomorphisms  and ${}_A\can$ is bijective, we conclude from the commutativity of the diagram
\eqref{comultpull} that the diagonal coaction $\Delta_{A\utimes_{{}_B}A}$ is an algebra homomorphism.
\end{proof}

\subsection{Braided join comodule algebras}

We begin by specializing the left Durdevic
 braiding \eqref{ldb} to left Galois objects. This means that now not only we assume that the left canonical map ${}_A\can$
 is bijective, but also that  the coaction-invariant subalgebra $^{coH}\!A$ is the ground field.  Therefore, we can simplify our
notation for the left  Hopf-Galois braided algebra to $A\utimes A$. To preserve the topological meaning of our
join construction in the commutative setting, from now on we specialize our ground field to be the field of complex numbers.

\begin{definition}\label{bjoin}
Let $H$ be a Hopf algebra over $\mathbb{C}$ and $A$ be a bicomodule algebra over $H$. 
Assume that $A$ is a left Galois object over $H$ and
 $A \, \utimes \, A$ is a left Hopf-Galois braided algebra. We call the
unital $\mathbb{C}$-algebra
$$
A\underset{H}{*}A:=\{x\in C([0,1])\otimes A\underline\otimes A\;|\; 
(\mathrm{ev}_0\otimes\id)(x)\in\mathbb{C}\otimes A\text{ and }
(\mathrm{ev}_1\otimes\id)(x)\in A\otimes \mathbb{C}\}
$$
the \emph{$H$-braided  join algebra of $A$}.
Here $\mathrm{ev}_r$ is the evaluation map at $r\in [0,1]$, i.e.\
$\mathrm{ev}_r(f)=f(r)$.
\end{definition}

\begin{lemma}
Let $A*_HA$ be the $H$-braided  join algebra of $A$.
Then the formula
$$
C([0,1])\otimes A\underline\otimes A\ni f\otimes a\otimes b\longmapsto 
f\otimes a_{(0)}\otimes b_{(0)}\otimes a_{(1)}b_{(1)}\in
C([0,1])\otimes A\underline\otimes A\otimes H
$$
restricts to $\Delta_{A*_HA}\colon A*_HA\to (A*_HA)\otimes H$ making $A*_HA$ a right 
$H$-comodule algebra.
\end{lemma}
\begin{proof}
Let $\sum_i f_i\ot a_i\ot b_i\in A*_HA$, i.e.\ $\sum_i f_i(0)\,a_i\ot b_i\in \mathbb{C}\ot A$ and
$\sum_i f_i(1)\,a_i\ot b_i\in A\ot \mathbb{C}$. Then 
\begin{gather}
(\mathrm{ev}_r\ot\id)\Big(\sum_i f_i\ot (a_i)_{(0)}\utimes (b_i)_{(0)}\ot (a_i)_{(1)}(b_i)_{(1)}\Big)\\
=\sum_i \big(f_i(r)a_i\big)_{(0)}\utimes (b_i)_{(0)}\ot \big(f_i(r)a_i\big)_{(1)}(b_i)_{(1)}\,.\nonumber
\end{gather}
For $r=0$ the above tensor belongs to $\mathbb{C}\ot A\ot H$, and for $r=1$ the above tensor belongs to $A\ot\mathbb{C}\ot H$.
\end{proof}

\subsection{Pullback structure and principality}

In order to compute the coaction-invariant subalgebra, and to show that the principality of the right $H$-coaction 
on $A$ implies the principality of the right diagonal $H$-coaction on $A*_HA$, we present
$A*_HA$ as a pullback of right $H$-comodule algebras. Define 
\begin{align} \label{A1} 
A_1&:= \{f\in C([0,\mbox{$\frac{1}{2}$}])\otimes A\underline\otimes A\;|\; 
(\mathrm{ev}_0\otimes\id)(f)\in\mathbb{C}\utimes A\},\\
A_2&:= \{g\in C([\mbox{$\frac{1}{2}$},1])\otimes A\underline\otimes A\;|\; 
(\mathrm{ev}_1\otimes\id)(g)\in A\utimes\mathbb{C}\}. \label{A2} 
\end{align}
Then  $A*_HA$  is isomorphic to the pullback of $A_1$ and $A_2$ 
over $A_{12} := A\utimes A$ along the right $H$-colinear evaluation maps 
\[ \label{maps}
\pi_1:=\mathrm{ev}_\frac{1}{2}\otimes\id\,:\,A_1 \longrightarrow A_{12},\qquad 
\pi_2:=\mathrm{ev}_\frac{1}{2}\otimes\id\,:\,A_2 \longrightarrow A_{12}. 
\]
By Lemma \ref{unbraid}, ${}_A\can$ is a right $H$-comodule algebra isomorphism 
${}_A\can: A\utimes A \ra H\ot A$.  Also, we have ${}_A\can (\mathbb{C}\utimes A)= \mathbb{C}\ot A$ and 
${}_A\can(A\utimes \mathbb{C})={}_A\Delta(A)$. 
Next we note that the  right $H$-comodule algebras $A_1$ and $A_2$ 
are  isomorphic to 
\begin{align} \label{B1} 
B_1&:= \{f\in C([0,\mbox{$\frac{1}{2}$}])\otimes H\otimes A\;|\; 
(\mathrm{ev}_0\otimes\id)(f)\in\mathbb{C}\ot A\},\\
B_2&:= \{g\in C([\mbox{$\frac{1}{2}$},1])\otimes H\otimes A\;|\; 
(\mathrm{ev}_1\otimes\id)(g)\in {}_A\Delta(A)\}, \label{B2}
\end{align}
respectively. 

Since $\Delta_A(a)=a\ot 1$ implies that $a\in\mathbb{C}$, we obtain
\begin{align} \label{B1inv} 
B_1^{\,\mathrm{co}H}&:= \{f\in C([0,\mbox{$\frac{1}{2}$}])\otimes H\otimes \bC\;|\; 
f(0)\in\bC\},\\
\label{B2inv}
B_2^{\,\mathrm{co}H}&:= \{g\in C([\mbox{$\frac{1}{2}$},1])\otimes H\otimes \bC\;|\; 
g(1)\in\bC\}.  
\end{align}
In both cases, these algebras are isomorphic to the unreduced cone of $H$. As a result the coaction-invariant subalgebra
of $A*_HA$
is  isomorphic to the unreduced suspension of $H$, i.e.
\[\label{coinv}
(A*_HA)^{\,\mathrm{co}H}\cong\Sigma H:=\{g\in C([0,1])\otimes H\;|\; g(0),g(1)\in\bC\}\,.
\]

\begin{lemma}  \label{Bp}
Let $H$ be a Hopf algebra with bijective antipode 
and  $A$ be a bicomodule algebra over $H$. Also, let $A$ be a left and \emph{right} Galois object over $H$,
 and $A \, \utimes \, A$ be a left Hopf-Galois braided algebra.
Then the right $H$-comodule algebras $B_1$ and $B_2$ are principal. 
\end{lemma}
\begin{proof}
To prove the lemma, it suffices to show the existence of  strong connections on $B_1$ and~$B_2$ \cite[p.~599]{hkmz11}. 
Note first that the right translation map for a Galois object over a Hopf algebra with bijective antipode is a strong connection.
Therefore, we will use the following notation $\can_A^{-1}(1\ot h)= h^\tone  \ot h^\ttwo$ for the right translation map. 
Let 
\begin{align}
& \ell_1 : H \lra B_1\ot B_1,\quad \ell_1(h) :=(1\ot 1 \ot h^\tone) \ot (1\ot 1 \ot h^\ttwo), \\
&  \ell_2 : H \lra B_2\ot B_2,\quad 
\ell_2(h) := (1\ot h^\tone{}_{(-1)} \ot h^\tone{}_{(0)})  \ot (1\ot h^\ttwo{}_{(-1)} \ot h^\ttwo{}_{(0)})  .
\end{align}
The unitality  of both $\ell_1$ and $\ell_2$ follows immediately from the unitality of the 
right translation map. 

Furthermore,  
\begin{align} \label{l1}
(\can_A\circ \ell_1)(h)&=1\ot 1 \ot h^\tone h^\ttwo{}_{(0)}  \ot h^\ttwo{}_{(1)}=1\ot 1 \ot  1\ot h,\\[4pt]
(\can_A\circ \ell_2)(h)&=         
1\ot h^\tone{\npt}_{(-1)}\, h^\ttwo{\npt}_{(-1)} \ot h^\tone{\npt}_{(0)}\,h^\ttwo{\npt}_{(0)}  \ot  h^\ttwo{\npt}_{(1)}\\
&=(\id\ot {}_A\Delta\ot\id)(1\ot h^\tone\,h^\ttwo{\npt}_{(0)}  \ot  h^\ttwo{\npt}_{(1)})
\nonumber\\
&= 1\ot 1\ot 1\ot h\,.  \nonumber
\end{align}

Finally, we verify the bicolinearity of $\ell_1$ and $\ell_2$. For $\ell_1$ it follows immediately from the bicolinearity of the right translation map.
For the right $H$-colinearity of $\ell_2$, we use the right $H$-colinearity of the right translation map to compute
\begin{align}
(\id\ot \Delta_{B_2}) (\ell_2 (h))&=
(1\ot h^\tone{}_{(-1)} \ot h^\tone{}_{(0)})  \ot (1\ot h^\ttwo{}_{(-1)} \ot h^\ttwo{}_{(0)}) \ot h^\ttwo{}_{(1)}
\nonumber\\ &=
(\id \ot{}_A\Delta\ot\id \ot {}_A\Delta\ot\id) (1\ot h^\tone \ot 1 \ot h^\ttwo{}_{(0)}\ot h^\ttwo{}_{(1)})
\nonumber\\ &=
(\id \ot{}_A\Delta\ot\id \ot {}_A\Delta\ot\id)
 (1\ot h_{(1)}{}^\tone \ot 1 \ot h_{(1)}{}^\ttwo \ot h_{(2)})
\nonumber\\ &=
(1\ot h_{(1)}{}^\tone{}_{(-1)} \ot h_{(1)}{}^\tone{}_{(0)})  \ot (1\ot h_{(1)}{}^\ttwo{}_{(-1)} \ot h_{(1)}{}^\ttwo{}_{(0)})\ot h_{(2)}
\nonumber\\ &=
\ell_2 (h_{(1)})\ot h_{(2)} = (\ell_2 \ot \id)(\Delta (h))\,.
\end{align}
Much in the same way, for the left $H$-colinearity of $\ell_2$, we use the left $H$-colinearity of the right translation map to compute
\begin{align}
(\Delta^L_{B_2}\ot\id) (\ell_2 (h))&=
S^{-1}(h^\tone{}_{(1)}) \ot (1\ot h^\tone{}_{(-1)} \ot h^\tone{}_{(0)})  \ot (1\ot h^\ttwo{}_{(-1)} \ot h^\ttwo{}_{(0)})
\nonumber\\ &=
(\id \ot\id\ot {}_A\Delta\ot\id \ot {}_A\Delta) \big( S^{-1}(h^\tone{}_{(1)})\ot 
1\ot h^\tone{}_{(0)} \ot 1 \ot h^\ttwo\big)
\nonumber\\ &=
(\id \ot\id\ot {}_A\Delta\ot\id \ot {}_A\Delta)
 (h_{(1)}\ot 1\ot h_{(2)}{}^\tone\ot 1 \ot h_{(2)}{}^\ttwo)
\nonumber\\ &=
h_{(1)}\ot (1\ot h_{(2)}{}^\tone{}_{(-1)} \ot h_{(2)}{}^\tone{}_{(0)})  \ot (1\ot h_{(2)}{}^\ttwo{}_{(-1)} \ot h_{(2)}{}^\ttwo{}_{(0)})
\nonumber\\ &=
h_{(1)}\ot \ell_2 (h_{(2)})  = (\id\ot\ell_2 )(\Delta (h))\,.
\end{align}
Summarizing, $\ell_1$ and $\ell_2$ are strong connections, and the lemma follows.
\end{proof}

We already know that the coaction-invariant subalgebra of $A*_HA$ is isomorphic to the unreduced suspension of~$H$~\eqref{coinv}.
Now, combining the above lemma with \cite[Lemma~3.2]{hkmz11} and the right $H$-comodule algebra isomorphisms
$B_i\cong A_i$, $i\in\{1,2\}$, we arrive at the main theorem of this paper:
\begin{theorem}\label{dangerous}
Let $H$ be a Hopf algebra with bijective antipode. Assume that  $A$ is a bicomodule algebra
and a left and right Galois object over $H$. 
Then the coaction
$$
\Delta_{A \underset{H}{*}  A}\,: \,A \underset{H}{*}  A\,\longrightarrow \,(A \underset{H}{*} A)\otimes H
$$
is \emph{principal}.
Furthermore,
 the coaction-invariant subalgebra 
 $(A*_HA)^{\mathrm{co}H}$ is isomorphic to the unreduced suspension of~$H$~\eqref{coinv}.
\end{theorem}

\section{*-Galois objects}\label{three}
\setcounter{equation}{0}

\subsection{*-structure}

Assume now that $H$ is a *-Hopf algebra. This means that $H$ is a Hopf algebra and a *-algebra such that
\[
(*\ot*)\circ \Delta =\Delta \circ *,  \quad 
*\circ S\circ *\circ S = \id\quad \text{and}\quad  
\varepsilon\circ * = \overline{\phantom{x} }\circ\varepsilon ,
\]
 where bar denotes the complex conjugation.

Much in the same way, we call $A$ a right $H$ *-comodule algebra iff it is a *-algebra and a right $H$-comodule algebra such that
\[
(*\ot*)\circ \Delta_A =\Delta_A \circ *\,.
\]
A left *-comodule algebra is defined in the same manner. 

Next,  we use the algebra isomorphism ${}_A\can : A\ot A\ra H\ot A$ (see Lemma~\ref{unbraid}) to pullback the natural
*-structure on $H\ot A$ (given by $(h\ot a)^* = h^* \ot a^*$) to obtain the following *-structure on the braided algebra $A\utimes A$:
\begin{align} \label{star}
(a\utimes b)^*:\! &= ({}_A\can^{-1} \circ (*\ot *) \circ {}_A\can)(a\utimes b)
\nonumber\\ &=
{a^*}_{(-1)}{}^{[1]} \utimes {a^*}_{(-1)}{}^{[2]} \, b^* \,{a^*}_{(0)}=(1\utimes b^*)\sbullet(a^*\otimes 1). 
\end{align}
Our goal now is to show:
\begin{proposition}
 If $A$ is an $H$ *-bicomodule algebra and
a left $H$-Galois object, then the $H$-braided join algebra $A*_HA$ is a right $H$ *-comodule algebra for the diagonal coaction.
\end{proposition}
\begin{proof}
With the complex conjugation in the first component and the aforementioned *-structure on
$A\utimes A$,  the algebra $C([0,1])\otimes A\underline\otimes A$ becomes 
a *-algebra. On the other hand, it follows from \eqref{star} that $(\bC\utimes A)^* = \bC\utimes A$ and 
$(A\utimes \bC)^* = A\utimes \bC$. Therefore, as evaluation maps are \mbox{*-homomorphisms}, the *-structure on
$C([0,1])\otimes A\underline\otimes A$  restricts to a *-structure on $A*_HA$. 

Furthermore, we know from Lemma~\ref{unbraid} that $\Delta_{A\utimes A}={}_A\can^{-1}\circ(\id\ot\Delta_A)\circ{}_A\can$. Since
all the involved maps are *-homomorphisms, so is~$\Delta_{A\utimes A}$.   Finally, since $\Delta_{A*_HA}$ is a restriction of
 $\id\ot\Delta_{A\utimes A}$, and $\Delta_{A\utimes A}$ is a *-homomorphism,
it follows that $\Delta_{A*_HA}$ is a *-homomorphism. 
\end{proof}

\begin{remark}{\rm
Although it is not needed for our immediate purposes, for the sake of completeness, let us prove the left-sided version
of Durdevic's formula relating the *-structure with the left translation map \cite[Section~2]{d-m96}. Let $H$ be a *-Hopf algebra, 
and $Q$ a left $H$ \mbox{*-comodule} algebra such that the left canonical map \eqref{can} is bijective. 
Then the left translation map (see~\eqref{ltrans})
satisfies
\[
\forall\; h\in H:\;\tau(h^*)=(h^*)^{[1]} \ot (h^*)^{[2]} = (S^{-1}(h))^{[2]\,*} \ot  (S^{-1}(h))^{[1]\,*} .
\]
To prove this, it suffices to show that ${}_Q\can$ applied to the right hand side gives $h^*\ot 1$. 
Using \eqref{lacl} in the second equality, we get 
\begin{align} 
&\big((S^{-1}(h))^{[2]\,*}\big){}_{(-1)} \ot \big((S^{-1}(h))^{[2]\,*}\big){}_{(0)} (S^{-1}(h))^{[1]\,*} \nonumber \\
&= \big((S^{-1}(h))^{[2]}{}_{(-1)}\big)^* \ot  \big((S^{-1}(h))^{[2]}{}_{(0)}\big)^* (S^{-1}(h))^{[1]\,*}\nonumber \\
&= \big(S (S^{-1}(h_{(1)}))\big)^*  \ot (S^{-1}(h_{(2)}))^{[2]\,*}\, (S^{-1}(h_{(2)}))^{[1]\,*} \nonumber\\
&= h_{(1)}^*\ot \big((S^{-1}(h_{(2)}))^{[1]} \, (S^{-1}(h_{(2)}))^{[2]}\big)^*\nonumber \\
&= h_{(1)}^*\ot \,\overline{\vare(h_{2})}\nonumber\\
&= h^* \ot 1. 
\end{align}
}
\end{remark}

\subsection{Noncommutative-torus algebra as a Galois object}

In this subsection, we take the algebra $\mathcal{O}(\T^2)$ of Laurent polynomials in two variables  as our *-Hopf algebra~$H$.
It is generated by
commuting unitaries $u$ and $v$, and the Hopf algebra structure is defined by
\[
\Delta(u) =u \ot u,\ \  \Delta(v)= v\ot v, \ \ \vare(u)=1=\vare(v),\ \ S(u)=u^*, \ \ S(v)=v^*. 
\]
 Next, let $\theta\in[0,1)$  and let $A:=\mathcal{O}(\T^2_\theta)$ denote the polynomial *-algebra of 
the noncommutative torus, i.e.\ the *-algebra
generated by unitary elements $U$ and $V$ satisfying the relation 
\[\label{theta}
UV=\mathrm{e}^{2\pi\mathrm{i}\theta} VU. 
\]
We define coactions $\Delta_\NT: \NT \ra \NT\ot \HB$  and ${}_\NT\Delta: \NT \ra \HB\ot \NT$   by  
\[
\Delta_\NT(U) := U\ot u,\ \  \Delta_\NT(V) := V\ot v, \ \ {}_\NT\Delta(U) := u\ot U,\ \  {}_\NT\Delta(V) := v\ot V. 
\]
These  coactions turn $\NT$ into an $H$ *-bicomodule algebra. 
Since
$\{ U^k V^l\,| \, k,l \in \bZ\}$ is a linear basis of $\NT$ (by the Diamond Lemma \cite[Theorem~1.2]{b-g78}), one sees immediately 
that \mbox{${}^{\mathrm{co}\HB}\!\NT=\bC=\NT^{\mathrm{co}\HB}$}. Furthermore, it is straightforward to check that 
the inverses of the left and right canonical maps  are respectively given by 
\[
{}_A\can^{-1}(u^kv^l\ot a) =  U^kV^l   \ot V^{-l}U^{-k} a, \quad
\can_A^{-1}(a\ot u^kv^l) = a V^{-l}U^{-k}\ot U^kV^l   . 
\]
Hence $\NT$ is a left and right Galois object over~$H$. As the antipode of $H$ is bijective, $A$ satisfies all assumptions
of Theorem~\ref{dangerous}.

Using \eqref{theta}, one easily verifies that
the braiding \eqref{ldb} reads 
\[
\NT\ot\NT\ni  U^kV^l\ot U^m V^n \longmapsto       
  \mathrm{e}^{2\pi\mathrm{i}\theta(kn-lm)}   U^m V^n \ot  U^kV^l  \in \NT\ot\NT\,.
\]
Now the product \eqref{braided_product} in 
$\NT\utimes\NT$ is determined by 
\begin{align}
(U^rV^s\utimes U^k V^l )\sbullet (U^mV^n\utimes U^a V^b) &= 
 \mathrm{e}^{2\pi\mathrm{i}\theta(kn-lm)} \, U^rV^s U^mV^n \utimes   U^k V^l U^a V^b \nonumber\\
 &=  \mathrm{e}^{2\pi\mathrm{i}\theta(kn-lm-sm-la)} \, U^{r+m}V^{s+n}  \utimes   U^{k+a} V^{l+b}, 
\end{align} 
where $r,s,k,l,m,n,a,b\in\bZ$.
One readily checks that the elements
\[ \label{UVUV}
U_L:= U\utimes 1,\quad V_L := V\utimes 1, \quad U_R:= 1\utimes U,\quad V_R := 1\utimes V,
\]
satisfy the relations 
\begin{align}
U_RU_L=U_LU_R,\quad V_RV_L=V_LV_R,&\quad
U_LV_L= \mathrm{e}^{2\pi\mathrm{i}\theta} V_LU_L, \quad
U_RV_R= \mathrm{e}^{2\pi\mathrm{i}\theta} V_RU_R, \label{qt1}\\
 U_R V_L =  \mathrm{e}^{2\pi\mathrm{i}\theta}  V_L U_R, &\quad
V_R U_L =  \mathrm{e}^{-2\pi\mathrm{i}\theta} U_L V_R\,.\label{qt2}
\end{align} 
It follows from \eqref{star} that $U_L$, $V_L$, $U_R$, $V_R$ are unitary. Furthermore, since they generate $A\utimes A$,
any element $y\in C([0,1])\otimes \NT \utimes \NT$ can be written as 
\[ \label{y}
y= \sum_{\text{finite}} f_{klmn}\otimes U_L^kV_L^l U_R^m V_R^n , \qquad f_{klmn}\in C([0,1]). 
\]
From $U_L^kV_L^l U_R^m V_R^n = U^kV^l \utimes U^m V^n$, we conclude that 
\begin{align} \label{AHA}
\NT\underset{\HB}{*}\NT=&\left\{\sum_{\text{finite}}  f_{klmn}\otimes U_L^kV_L^l U_R^m V_R^n
 \in C([0,1])\ot \NT\utimes\NT \;\Big|\; \right.\nonumber\\    
 &\phantom{xx} k,l,m,n \in\bZ, \;
 f_{klmn}(0)=0 \text{ for } (k,l)\neq (0,0), \;
  f_{klmn}(1)=0 \text{ for } (m,n)\neq (0,0)\Big\}. 
\end{align}

Finally, the diagonal coaction $\Delta_{A*_HA} : \NT*_H\NT\ra (\NT*_H\NT)\ot \HB$ is determined by 
\[
\Delta_{\NT\underset{H}{*}\NT}(f\ot U_L^kV_L^l U_R^m V_R^n) 
= f\ot U_L^kV_L^l U_R^m V_R^n\ot u^{k+m} v^{l+n}. 
\]
By Theorem~\ref{dangerous}, the above coaction is principal (admits a strong connection), and the coaction-invariant subalgebra
$(\NT*_\HB\NT)^{\mathrm{co}\HB}$ can be viewed 
as an algebra of functions on the unreduced suspension of the classical torus. 
Explicitly, we have
$$
(\NT\underset{\HB}{*}\NT)^{\mathrm{co}\HB} 
=\left\{\sum_{\text{finite}}  g_{kl}\otimes X^k Y^l
 \in \NT\underset{\HB}{*}\NT    \;\Big|\;      
   g_{kl}(0)=0 =g_{kl}(1)\text{ for } (k,l)\neq (0,0), \ k,l \in\bZ \right\},
$$
where 
$X:= U_LU_R^{*} = U\utimes U^*$ and  $Y:= V_L V_R^{*}=V\utimes V^{*}$
are commuting unitaries.

To end with, let us note that, as the Hopf algebra $H$ is commutative, the diagonal coaction
$A\ot A\to A\ot A\ot H$ is an algebra homomorphism already for the trivial braiding (the flip). However, 
for the non-braided tensor algebra~$A\ot A$, the left canonical map ${}_A\can$ is no longer an algebra homomorphism:
\begin{align}
{}_A\can \big((1\ot U)(V\ot 1)\big)&={}_A\can (V\ot U)=v\ot VU\nonumber\\
&\neq v\ot UV=(1\ot U)(v\ot V)={}_A\can(1\ot U){}_A\can(V\ot 1).
\end{align}
The braided algebra $A\utimes A$ is ``more noncommutative" than $A\ot A$ in the sense that
the relations \eqref{qt1} among generators are the same in both cases, but the relations \eqref{qt2} simplify
to the commutativity of generators for~$A\ot A$.

\section{Finite quantum coverings}\label{four}
\setcounter{equation}{0}

In this section, first we show that for any finite-dimensional Hopf algebra $H$, the anti-Drinfeld double $A(H)$ is a bicomodule algebra and
 a left and right Galois object over the Drinfeld double Hopf algebra~$D(H)$. Then we apply our braided noncommutative
join construction to the aforementioned Galois object for a concrete 9-dimensional Hopf algebra~$H$.

\subsection{(Anti-)Drinfeld doubles}

Recall that for any finite-dimensional Hopf algebra $H$, one can define the Drinfeld double Hopf algebra $D(H):=H^*\ot H$ by
the following formulas for multiplication and comultiplication~\cite{d-vg87}:
\begin{align}
(\varphi \otimes h)(\varphi' \otimes h') &=
\varphi'{}_{(1)}(S^{-1}(h_{(3)})) \varphi'{}_{(3)}(h^{(1)}) \; 
\varphi \varphi'{}_{(2)} \otimes h_{(2)} h'\,,\\ \label{d-comult}
\Delta(\varphi \otimes h)&=\varphi_{(2)} \otimes h_{(1)}\ot\varphi_{(1)} \otimes h_{(2)}\,.
\end{align}
Here $H^*$ is the dual Hopf algebra, and the Heyneman-Sweedler indices refer to the coalgebra structures on $H^*$ and~$H$.
Therefore, as a coalgebra, $D(H)=(H^*)^{\mathrm{cop}}\ot H$. 

Much in the same way, one can define the anti-Drinfeld double right $D(H)$-comodule algebra $A(H):=H^*\ot H$  by
the following formulas for multiplication and coaction respectively~\cite{hkrs04a}:
\begin{align}\label{a-mult}
(\varphi \otimes h)(\varphi' \otimes h') &=
\varphi'{}_{(1)}(S^{-1}(h_{(3)})) \varphi'{}_{(3)}(S^2(h_{(1)})) \; 
\varphi \varphi'{}_{(2)} \otimes h_{(2)} h'\,,\\
\label{rco}
\Delta_{A(H)}(\varphi \otimes h)&=\varphi_{(2)} \otimes h_{(1)}\ot\varphi_{(1)} \otimes h_{(2)}\,.
\end{align}
Note that, since the formula for the right coaction is the same as the formula for the comultiplication and
as a vector space $A(H)=D(H)$, we immediately conclude that $A(H)$ is a right $D(H)$-Galois object.
This reflects the combination of the following facts:
 any Yetter-Drinfeld module over $H$ is a module over the Drinfeld double $D(H)$, 
any anti-Yetter-Drinfeld module over $H$ is a module over the anti-Drinfeld double $A(H)$,
and the tensor product of an anti-Yetter-Drinfeld module with a Yetter-Drinfeld module is an anti-Yetter-Drinfeld module
(see \cite{hkrs04a} for details).

Next, let us observe that the formula
\[\label{leftco}
{}_{A(H)}\Delta(\psi \otimes k)=\psi_{(2)} \otimes S^2(k_{(1)})\ot\psi_{(1)} \otimes k_{(2)}\,
\]
defines a left $D(H)$-coaction on $A(H)$, which commutes with the above defined right coaction~$\Delta_{A(H)}$. Also,
since the comultiplication formula \eqref{d-comult} differs from the left coaction  formula \eqref{leftco} only by an automorphism
$S^2$, the coaction invariant subalgebra is trivial: \mbox{${}^{\mathrm{co} D(H)}A(H)=\mathbb{C}$}. Thus
to arrive at the assumptions of our main result (Theorem~\ref{dangerous}), it suffices to show that ${}_{A(H)}\Delta$ is an
algebra homomorphism. (The antipode of any finite-dimensional Hopf algebra is bijective~\cite{ls69}.)

To this end, note first that $\varphi$ and $h'$ do not play an essential role in the multiplication formula~\eqref{a-mult}. One can easily
check that to prove that ${}_{A(H)}\Delta$ is an
algebra homomorphism, one can restrict to $\varphi=\varepsilon$ and $h'=1$. Now we compute
\begin{align}
{}_{A(H)}\Delta\big((\vare \otimes h)(\varphi' \otimes 1)\big) &=
{}_{A(H)}\Delta\big(\varphi'{}_{(1)}(S^{-1}(h_{(3)})) \varphi'{}_{(3)}(S^2(h_{(1)})) \; 
 \varphi'{}_{(2)} \otimes h_{(2)}\big) \nonumber\\ 
&=\big(\varphi'{}_{(1)}(S^{-1}(h_{(4)})) \varphi'{}_{(4)}(S^2(h_{(1)})) \; 
\varphi'{}_{(3)} \otimes  S^2(h_{(2)}) \big)
\otimes \big(\varphi'{}_{(2)} \otimes h_{(3)}\big)\,.
\end{align}
On the other hand, we compute
\begin{align*}
&{}_{A(H)}\Delta(\vare \otimes h) {}_{A(H)}\Delta(\varphi' \otimes 1) \\
&=
\big((\vare \otimes S^2(h_{(1)}))(\varphi'{}_{(2)} \otimes 1)\big) \otimes
\big( ( \vare \otimes h_{(2)}) (\varphi'{}_{(1)} \otimes 1)\big)\\
&=
\big(\varphi'{}_{(2)}(S(h_{(3)})) \varphi'{}_{(4)}(S^2(h_{(1)})) \; 
\varphi'{}_{(3)} \otimes  S^2(h_{(2)}) \big) \otimes
\big( ( \vare \otimes h_{(4)}) (\varphi'{}_{(1)} \otimes 1)\big)
 \nonumber \\
&=
\big(\varphi'{}_{(4)}(S(h_{(3)})) \varphi'{}_{(6)}(S^2(h_{(1)})) \; 
\varphi'{}_{(5)} \otimes  S^2(h_{(2)}) \big) 
\otimes
\big(\varphi'{}_{(1)}(S^{-1}(h_{(6)})) \varphi'{}_{(3)}(S^2(h_{(4)})) \; 
\varphi'{}_{(2)} \otimes h_{(5)} \big) 
 \nonumber \\
&=
\varphi'{}_{(1)}(S^{-1}(h_{(6)}))\varphi'{}_{(3)}(S^2(h_{(4)}))\varphi'{}_{(4)}(S(h_{(3)}))\varphi'{}_{(6)}(S^2(h_{(1)}))
\big(\varphi'{}_{(5)} \otimes  S^2(h_{(2)}) \big) \ot \big(\varphi'{}_{(2)} \otimes h_{(5)}
\big) \\
&=
\varphi'{}_{(1)}(S^{-1}(h_{(6)}))\varphi'{}_{(3)}\big(S\big(h_{(3)}S(h_{(4)})\big)\big)\varphi'{}_{(5)}(S^2(h_{(1)}))
\big(\varphi'{}_{(4)} \otimes  S^2(h_{(2)}) \big) \ot \big(\varphi'{}_{(2)} \otimes h_{(5)}
\big) \\
&=
\varphi'{}_{(1)}(S^{-1}(h_{(4)}))\varphi'{}_{(3)}(1)\varphi'{}_{(5)}(S^2(h_{(1)}))
\big(\varphi'{}_{(4)} \otimes  S^2(h_{(2)}) \big) \ot \big(\varphi'{}_{(2)} \otimes h_{(3)}
\big) \\
&=
\varphi'{}_{(1)}(S^{-1}(h_{(4)})) \varphi'{}_{(4)}(S^2(h_{(1)})) \; 
\big(\varphi'{}_{(3)} \otimes  S^2(h_{(2)}) \big)
\otimes \big(\varphi'{}_{(2)} \otimes h_{(3)}\big)\,.
\end{align*}
Hence ${}_{A(H)}\Delta$ is an algebra homomorphism, as needed. Summarizing, we have arrived at:
\begin{theorem}
Let $H$ be a finite-dimensional Hopf algebra. Then the anti-Drinfeld double $A(H)$ is a bicomodule algebra
and a left and right Galois object over the Drinfeld double $D(H)$ for coactions given by the formulas
\eqref{leftco} and~\eqref{rco}.
\end{theorem}

\subsection{A finite quantum subgroup of $SL_{e^{2\pi i/3}}(2)$}

Let $q:=e^{2\pi i/3}$, and let  $\HA$ denote the Hopf algebra generated by $a$ and $b$ 
satisfying the relations  
\[ \label{HArel}
a^3=1,\qquad  b^3=0, \qquad ab = qba. 
\]
The comultiplication $\Delta$, counit $\vare$, and antipode $S$  are respectively given by 
\[
\Delta(a)= a\otimes a, \ \ \Delta(b) = a\otimes b + b\otimes a^2 , \ \  \vare(a)=1,\ \  \vare(b)=0 \ \ 
S(a)=a^2,\ \ S(b)=-q^2 b. 
\]
The set $\{b^na^m\}_{n,m=0,1,2}$ is a linear basis of $\HA$ \cite[Proposition~4.2]{dhs99}. 

The structure of the dual Hopf algebra $\HA^*$ and its pairing with $H$ can be deduced from~\cite{dns98}.  
We use  generators $k$ and $f$ of $\HA^*$ that in terms of generators used in \cite{dns98} can be written as follows:
 $k$ is the equivalence class of the grouplike generator of $U_q(sl(2))$ and $f:= q^2kx_-$, where $x_-$ 
is the equivalence class of  $X_-\in U_q(sl(2))$. Our generators satisfy the relations 
\[\label{HA*rel}
k^3 =1,\quad f^3=0,\quad fk=qkf. 
\]
The coproduct, counit  and antipode are respectively given by 
\[
\Delta(k)=k\ot k,\ \ \Delta(f)= f\ot 1 + k\ot f,\ \  \vare(k)=1,\  \ \vare(f)=0,\ \ S(k)=k^2, \ \ S(f)=-k^2f. 
\]
The formulas 
\[
k(a) := q,\quad k(b):=0,\quad f(a):=0,\quad  f(b):=1    	
\]
determine a non-degenerate pairing between $H^*$ and $H$. 
 
The Drinfeld double $D(\HA)$, as an algebra, is generated by 
\[
K:= k\ot 1, \quad F:= f\ot 1,\quad A:= 1\ot a,\quad B:= 1\ot b, 
\]
where $K$ and $F$ satisfy the same relations \eqref{HA*rel} as $k$ and $f$, and $A$ and $B$ satisfy 
the same relations \eqref{HArel} as  $a$ and $b$. They also fulfill the cross relations
\begin{align}
AK=KA,\quad AF=q^2FA,\quad BK=q^2 KB, \quad
BF= q FB + qKA^2 - qA. 
\end{align}
The coproduct, counit  and antipode are respectively determined by 
\begin{align}
\Delta(A)=A\ot A, \quad \Delta(B)= A\ot B + B\ot A^2, &\quad \Delta(K)= K\ot K,\quad\Delta(F)= 1 \ot F + F\ot K, 
\nonumber\\
\vare(A)=1=\vare(K),&\quad \vare(B)=\vare(F)=0,\nonumber\\
S(A)=A^2,\quad S(B)=-q^2B,&\quad S(K)=K^2,\quad S(F)= -FK^2. 
\end{align} 

For the anti-Drinfeld double $A(H)$ we define analogous generators:
\[
\tK:= k\ot 1, \quad  \tF:= f\ot 1,\quad  \tA:= 1\ot a,\quad  \tB:= 1\ot b. 
\]
It follows from  \eqref{a-mult} that  $\tK$ and $\tF$ satisfy the same relations as 
 $k$ and $f$, and $\tA$ and $\tB$  fulfill the same relations as $a$ and $b$. However, 
the cross relations now become
\begin{align}
\tA\tK=\tK\tA,\quad \tA\tF=q^2\tF\tA,\quad \tB\tK=q^2 \tK\tB, \quad
\tB\tF= q \tF\tB + q^2\tK\tA^2 - q\tA. 
\end{align}
The left and right  $D(H)$-coactions \eqref{leftco} and \eqref{rco} in terms of generators are
\begin{align}
{}_{D(\HA)}\Delta(\tA)&=  A\ot\tA, &\Delta_{D(\HA)} (\tA)&=\tA\ot A, \\ 
{}_{D(\HA)}\Delta(\tB)&=  A\ot\tB + q B \ot \tA^2, & \Delta_{D(\HA)}(\tB)&= \tA\ot B + \tB\ot A^2, \\
{}_{D(\HA)}\Delta(\tK)&=  K\ot\tK,  & \Delta_{D(\HA)}(\tK)&= \tK\ot K, \\
{}_{D(\HA)}\Delta(\tF)&=  1\ot\tF + F \ot \tK, &  \Delta_{D(\HA)}(\tF)&= 1 \ot F + \tF\ot K.  
\end{align}
%

Furthermore, there is an algebra isomorphism $\chi: A(\HA) \ra D(\HA)$ given by 
\[
\chi(\tA)=A,\quad \chi(\tB)=qB,\quad  \chi(\tK)=q^2 K,\quad \chi(\tF)= q^2 F.
\]
A direct calculation shows that ${}_{A(\HA)}\Delta= (\id \ot \chi^{-1})\circ \Delta\circ \chi$.  
Hence
\[ \label{cLinv}
{}_{A(H)}\can^{-1} (a\ot p):= \chi^{-1}(a_{(1)}) \ot \chi^{-1}(S(a_{(2)}))\,p.
\]
Indeed, applying the bijection ${}_{A(H)}\can$ to the right hand side of this equality yields
\[
a_{(1)}\ot\chi^{-1}(a_{(2)}) \chi^{-1}(S(a_{(3)}))\,p=a\ot p,
\]
as needed.

Our next step is to unravel the structure of the left Hopf-Galois braided algebra $A(H)\utimes A(H)$.
To this end, we choose its generators as follows:
\begin{align}\label{obvious}
A_L := \tA \utimes 1, \quad B_L := \tB \utimes 1,\quad K_L :=\tK \utimes 1,\quad F_L := \tF\utimes 1,\nonumber\\
A_R :=  1\utimes \tA , \quad B_R := 1\utimes \tB,\quad K_R := 1 \utimes \tK,\quad F_R :=  1 \utimes \tF. 
\end{align} 
Each of the sets of generators $\{ A_L,B_L,K_L,F_L\}$ and $\{ A_R,B_R,K_R,F_R\}$ 
satisfies  the commutation relations of $A(H)$, and from \eqref{braided_product} and \eqref{cLinv} we 
infer the cross relations: 
\begin{align}
&A_R A_L = A_L A_R,\qquad  B_R A_L = q^2 A_L B_R, \quad &&K_R A_L = A_L K_R,\qquad F_R A_L = q A_L F_R, \\
&A_R B_L = B_L A_R +(1-q^2)A_LB_R,\quad  &&B_RB_L = qB_LB_R+ (1-q)A_LA_RB_R^2,\\
&K_RB_L= B_L K_R +(q-1)A_LA_R^2B_RK_R,\quad &&F_R B_L = q^2B_L F_R  - q A_LA_RK_R  + A_L,\\
& A_R K_L = K_L A_R,\qquad B_R K_L = q^2 K_L ,\quad &&K_R K_L = K_L K_R,\qquad F_RK_L = q K_L F_R , \\ 
& A_R F_L = F_L A_R + (1-q) A_R F_R,\quad &&B_R F_L = q^2 F_L B_R - q A_R^2K_R + A_R, \\
& K_RF_L= F_L K_R +(1-q) K_R F_R, \quad &&F_R F_L = q F_L F_R + (1-q) F_R^2. 
\end{align}
Furthermore, since $A(H)\utimes A(H)\cong H^*\ot H\ot H^*\ot H$ as a vector space, the set
\[
\{ A_L^{n_1} B_L^{n_2} K_L^{n_3} F_L^{n_4} A_R^{n_5} B_R^{n_6} K_R^{n_7} F_R^{n_8} 
 \; |\; {n_1},\ldots,{n_8}\in\{ 0,1,2\} \}
\]
is a linear basis  of $A(H) \utimes A(H)$. 
Using this basis and remembering \eqref{obvious}, 
any element $X$ of  $C([0,1])\ot A(H) \utimes A(H)$ can be written as
\[ \label{X}
X= \sum_{n_1, \ldots, n_8 =0}^2  
f_{n_1,\ldots,n_8}\otimes \tA^{n_1} \tB^{n_2} \tK^{n_3} \tF^{n_4} \utimes  \tA^{n_5} \tB^{n_6} \tK^{n_7} \tF^{n_8} , 
\quad f_{n_1,\ldots,n_8}\in C([0,1]). 
\]
Hence 
\begin{align*}
&A(\HA)\!\underset{D(\HA)}{*}\!A(\HA)
=\Big\{\sum_{n_1,\ldots,n_8=0}^2\!\!\!  f_{n_1,\ldots,n_8}\otimes 
A_L^{n_1} B_L^{n_2} K_L^{n_3} F_L^{n_4} A_R^{n_5} B_R^{n_6} K_R^{n_7} F_R^{n_8} \; \Big|\;
\text{ all } f_{n_1,\ldots,n_8}\in C([0,1]),
\nonumber\\ 
&f_{n_1,\ldots,n_8}(0)=0 \text{ for } (n_1,\npt n_2,\npt n_3,\npt n_4)\neq (0,0,0,0), \ \,
f_{n_1,\ldots,n_8}(1)=0 \text{ for } (n_5,\npt n_6,\npt n_7,\npt n_8)\neq (0,0,0,0)
\Big\}. 
\end{align*}

For an explicit description of  the coaction-invariant subalgebra $(A(\HA){*}_{D(\HA)} A(\HA))^{\mathrm{co}D(\HA)}$, 
we use the fact that, by Lemma~\ref{unbraid}, the left canonical map ${}_{A(H)}\can$ is an isomorphism of  right 
\mbox{$D(H)$-comodule} algebras.  
This allows us to conclude that $\{a_L^jb_L^lk_L^m f_L^n\;|\;j,l,m,n\in\{0,1,2\}\}$, where 
\begin{align*}
a_L &:= {}_{A(H)}\can^{-1}(A\ot 1)  = \tA\utimes \tA^2, &&b_L:= {}_{A(H)}\can^{-1}(B\ot 1) = -q \tA \utimes \tB + q^2 \tB \utimes \tA, 
\\
k_L &:= {}_{A(H)}\can^{-1}(K\ot 1)  =  \tK\utimes \tK^2, &&f_L:= {}_{A(H)}\can^{-1}(F\ot 1) = -1\utimes \tF\tK^2 + \tF \utimes \tK^2, 
\end{align*}
 is a basis of  the coaction-invariant subalgebra 
\[
(A(H)\utimes A(H))^{\mathrm{co}D(\HA)}\cong D(H)\ot A(H)^{\mathrm{co}D(\HA)}=D(H)\ot\mathbb{C}.
\]
\noindent
Thus we obtain the following explicit description of the coaction-invariant subalgebra
\begin{align}
\Big(A(H)\!\underset{D(\HA)}{*}\!A(H)\Big)^{\mathrm{co}D(\HA)} = 
\Big\{&\sum_{j,l,m,n=0}^2 \!\!\! g_{jlmn}  \ot a_L^jb_L^lk_L^m f_L^n\;\Big|\;  \text{ all } g_{jlmn}  \in C([0,1]), \nonumber\\
& g_{jlmn} (0)=0=g_{jlmn} (1) \text{ for }  (j,l,m,n)\neq(0,0,0,0) \Big\} .
\end{align}
Since the generators $a_L,b_L,k_L,f_L$ satisfy the same 
commutation relations as  
the generators $A,B,K,F$ of $D(H)$, it is now evident that the coaction-invariant subalgebra 
is isomorphic to 
the unreduced suspension of $D(H)$, as claimed in Theorem~\ref{dangerous}.

\noindent{\bf Acknowledgments.}
It is a pleasure to thank S.~L.~Woronowicz 
for interesting discussions. Ludwik D\k{a}browski was  partially supported 
by the PRIN 2010-11 grant ``Operator Algebras,
 Noncommutative Geometry and
Applications" and WCMCS (Warsaw). He also
  gratefully ack\-nowledges the hospitality of ESI (Vienna), IHES (Bures-sur-Yvette) and IMPAN (Warsaw).
Tom Hadfield was financed via the EU Transfer of Knowledge contract 
MKTD-CT-2004-509794. \mbox{Piotr} M.\ Hajac was  partially supported by the NCN grant 2011/01/B/ST1/06474. 
Elmar \mbox{Wagner} was partially sponsored by WCMCS, IMPAN (Warsaw) and CIC-UMSNH (Morelia).


\begin{thebibliography}{BEEK72a}
\bibitem[B-G78]{b-g78}
G.\ Bergman,
\emph{The diamond lemma for ring theory},
Adv. in Math.  {\bf 29} (1978), 178--218.

\bibitem[B-GE93]{b-ge93}
G.~E. Bredon, \emph{Topology and geometry}, Graduate Texts in Mathematics 
139, Springer-Verlag, New York, 1993.

\bibitem[BH04]{bh04}
T.~Brzezi\'nski, P. M. Hajac, 
\emph{The Chern-Galois character}, C. R. Math. Acad. Sci. Paris {\bf 338} (2004), 113--116.

\bibitem[C-S98]{c-s98}
S.~Caenepeel, \emph{Brauer Groups, Hopf Algebras and Galois Theory}, 
in: K-Monographs in Math., vol. 4, Kluwer Academic, Dordrecht, 1998.

\bibitem[CM08]{cm08}
A.~Connes, M. Marcolli, \emph{Noncommutative geometry, quantum fields and motives}, 
American Mathematical Society Colloquium Publications, 55. American Mathematical Society, 
Providence, RI; Hindustan Book Agency, New Delhi, 2008.

\bibitem[DDHW]{ddhw}    
L.~D\k abrowski, K.~De Commer, P. M. Hajac, E.~Wagner,
\emph{Principal coactions on unreduced nuclear couple suspensions of unital C*-algebras}, in preparation.

\bibitem[DHH]{dhh}    
L.~D\k abrowski, T.~Hadfield, P. M. Hajac, 
\emph{Noncommutative join constructions}, arXiv:1407.6020.

\bibitem[DHS99]{dhs99}    
L.~D\k abrowski, P. M. Hajac, P. Siniscalco, 
\emph{Explicit Hopf-Galois description of $SL_{e^{2i\pi /3}}(2)$-induced Frobenius homomorphisms}, 
Proceedings of the ISI GUCCIA Workshop, 
Horizons in World Physics {\bf 226} (1999), 279--298.

\bibitem[DHW]{dhw}    
L.~D\k abrowski,  P. M. Hajac, E.~Wagner,
\emph{Braided join comodule algebras from homogeneous coactions}, in preparation.

\bibitem[DNS98]{dns98}    
L.~D\k abrowski, F. Nesti, P. Siniscalco, 
\emph{A finite quantum symmetry of M(3,C)} Internat. J. Modern Phys. {\bf  A 13} (1998), 4147--4161.

\bibitem[DS13]{ds13}   
L.~D\k abrowski, A. Sitarz, 
\emph{Noncommutative circle bundles and new Dirac operators}, 
Commun. Math. Phys. {\bf 318} (2013), 111--130. 

\bibitem[DSZ14]{dsz14}    
L.~D\k abrowski, A. Sitarz, A. Zucca, 
\emph{Dirac operator on noncommutative principal circle bundles}, 
Int. J. Geom. Methods Mod. Phys.  {\bf 11} (2014), 1450012.

\bibitem[DZ]{dz}  
L.~D\k abrowski, A. Zucca,  
\emph{Dirac operators on noncommutative principal torus bundles}, 
arXiv: 1308.4738. 

\bibitem[DT89]{dt89}    
Y.~Doi, M. Takeuchi, 
\emph{Hopf-Galois extensions of algebras, the Miyashita-Ulbrich action, and Azumaya algebras}, 
J. Algebra {\bf 121} (1989), 488--516.

\bibitem[D-VG87]{d-vg87}  
V.~G. Drinfeld, \emph{Quantum groups}, Proceedings of the International Congress of Mathematicians, Vol. 1, 2 (Berkeley, Calif., 1986), 798--820, 
Amer. Math. Soc., Providence, RI, 1987.

\bibitem[D-M96]{d-m96} M.~Durdevic, 
\emph{Quantum gauge transformations and braided structure 
on quantum principal bundles},  
arXiv:q-alg/9605010.

\bibitem[HKRS04a]{hkrs04a}
P.~M. Hajac, M. Khalkhali, B. Rangipour, Y. Sommerh\"auser, 
\emph{Stable anti-Yetter-Drinfeld modules}, 
C. R. Acad. Sci. Paris, Ser. I {\bf 338} (2004), 587--590.

\bibitem[HKRS04b]{hkrs04b}
P.~M. Hajac, M. Khalkhali, B. Rangipour, Y. Sommerh\"auser, 
\emph{Hopf-cyclic homology and cohomology with coefficients}, 
C. R. Acad. Sci. Paris, Ser. I {\bf 338} (2004), 667--672.

\bibitem[HKMZ11]{hkmz11}
P.~M. Hajac, U.\ Kr\"ahmer, R.\ Matthes, B.\ Zieli\'nski, 
\emph{Piecewise principal comodule algebras},
J. Noncommut. Geom. {\bf 5} (2011), 591--614. 
 

\bibitem[LS69]{ls69} 
R.~G. Larson, M.~E. Sweedler, 
\emph{An associative orthogonal bilinear form for Hopf algebras}, 
Amer. J. Math. {\bf 91} (1969), 75--94. 


\bibitem[M-J56]{m-j56} 
J.~Milnor, \emph{Construction of universal bundles. II}, Ann. of Math. {\bf 63} (1956), 430--436.

\bibitem[R-MA90]{r-ma90}
M.~A. Rieffel, \emph{Noncommutative tori ---  case study of noncommutative differentiable manifolds},  
Geometric and topological invariants of elliptic operators (Brunswick, ME, 1988), 191--211, 
Contemp. Math. {\bf 105}, Amer. Math. Soc., Providence, RI, 1990.

\bibitem[SS05]{ss05}
P.~Schauenburg, H.-J. Schneider,
\emph{On generalized Hopf Galois extensions},
J. Pure Appl. Algebra {\bf 202} (2005), 168--194. 


\bibitem[S-HJ90]{s-hj90}
H.-J. Schneider,
\emph{Representation theory of Hopf Galois extensions},
Israel J. Math.  {\bf 72} (1990), 196--231.
 

\bibitem[W-SL98]{w-sl98} 
S.~L. Woronowicz, 
\emph{Compact quantum groups. Sym\'etries quantiques (Les Houches, 1995)}, 
845--884, North-Holland, Amsterdam, 1998. 


\end{thebibliography}
\end{document}